\documentclass[10pt,reqno]{amsart}
\usepackage{amssymb,amsfonts}
\usepackage[all,arc]{xy}
\usepackage{enumerate}
\usepackage{mathrsfs}
\usepackage{geometry}
\usepackage{amsmath}
\usepackage{latexsym, bm}
\usepackage{indentfirst}
\usepackage{CJK}
\usepackage{amsthm}
\usepackage{hyperref}
\usepackage{cite}
\usepackage{lipsum}

\DeclareMathOperator{\tr}{tr}

\newcommand{\vp}{\varphi}
\newcommand{\pa}{\partial}

\newcommand{\na}{\nabla}
\newcommand{\al}{\alpha}

\newcommand{\wtl}{\widetilde}
\newcommand{\ov}{\overline}
\newcommand{\la}{\lambda}

\newcommand{\lra}{\longrightarrow}

\newtheorem{thm}{Theorem}[section]
\newtheorem{cor}[thm]{Corollary}

\newtheorem{prop}[thm]{Proposition}

\theoremstyle{definition}
\newtheorem{defn}[thm]{Definition}
\newtheorem{remk}[thm]{Remark}

\numberwithin{equation}{section}

\begin{document}
\title{Schwarz lemma: the case of equality and an extension}
\author{Haojie Chen, Xiaolan Nie}
\address{Department of Mathematics, Zhejiang Normal University, Jinhua Zhejiang, 321004, China}
\email{chj@zjnu.edu.cn, \ nie@zjnu.edu.cn }
\date{}
\maketitle

\vspace*{-.5cm}

\begin{center}
\textit{Dedicated to Professor Peter Li on the occasion of his 70th birthday}
\end{center}

\begin{abstract}
We prove two results related to the Schwarz lemma in complex geometry. First, we show that if the inequality in the Schwarz lemmata of Yau, Royden and Tosatti becomes equality at one point, then the equality holds on the whole manifold.  In particular, the holomorphic map is totally geodesic and has constant rank. In the second part, we study the holomorphic sectional curvature on an almost Hermitian manifold and establish a Schwarz lemma in terms of holomorphic sectional curvatures in almost Hermitian setting. 



\end{abstract}

\section{Introduction} 
The Ahlfors-Schwarz lemma \cite{Ah} is one of the most important results in complex analysis and differential geometry. It says that if $f$ is a holomorphic map from the unit disk $\mathbb D$ into a Riemann surface with a conformal metric $ds^2$ with Gaussian curvature at most $-1$, then $f$ is distance decreasing in the sense that 
\begin{align} \label{1.1} f^*ds^2\leq ds_{\mathbb D}^2,\end{align} where $ds_{\mathbb D}^2$ is the Poincar\'{e} metric on $\mathbb D$. In \cite{Hein}, Heins proved the strong form of the Ahlfors-Schwarz lemma: if the equality of $(\ref{1.1})$ holds at one point in $\mathbb D$, then it holds on $\mathbb D$ and $f$ is an isometric immersion. This generalizes the equality case of the classical Schwarz lemma. Later, the result was also proved by Royden \cite{RO2} and Minda \cite{Mi} using different methods (see also e.g. \cite{BK},\cite{LWT},\cite{BKO} for the boundary rigidity of the Schwarz lemma). Among other applications, the strong form of the Ahlfors-Schwarz lemma has been applied to study the Bloch constant problem \cite{Hein}, which is one of the most intriguing open problems in complex analysis of one variable.

{\let\thefootnote\relax\footnotetext{This research is partially supported by NSFC grants No. 11901530, No. 11801516, No. 12071161 and Zhejiang Provincial NSF grant No. LY19A010017.}}

In complex geometry, the Ahlfors-Schwarz lemma has been generalized to higher dimensional manifolds and to more general settings, by the work of Chern \cite{Chern}, Lu \cite{Lu}, Yau \cite{Ya1}, Chen-Cheng-Look \cite{CCL}, Royden \cite{RO1}, Tosatti \cite{T1} and Ni \cite{Ni1}\cite{Ni2} etc. (see also \cite{Ko} and references therein). In the first part of the paper, we study the equality case of the Schwarz lemmata obtained by Yau, Royden and Tosatti. To be more precise, Yau \cite{Ya1} proves the following.
\begin{thm} [Yau] \label{Yau} Let $(M, g)$ be a complete K\"ahler manifold and $(N, h)$ be a Hermitian manifold. Assume that the Ricci curvature of $M$ is bounded from below by $k\leq 0$ and the holomorphic bisectional curvature of $N$ is bounded from above by $K<0$. Then for any holomorphic map $f:M\rightarrow N$, 
\begin{align}\label{Yau2}
\|df\|^2\leq \dfrac{k}{K},
\end{align}
where $\|df\|^2=\tr_g(f^*h)$ is the trace of the pull back metric $f^*h$.
\end{thm}
In \cite{RO1}, Royden improves Yau's Schwarz lemma when $N$ is K\"ahler.
\begin{thm} [Royden] \label{Ro1} Assume that $N$ is K\"ahler. In Theorem \ref{Yau}, if the condition of the holomorphic bisectional curvature of $N$ is replaced by that the holomorphic sectional curvature of $N$ is bounded from above by $K<0$, then
\begin{align}\label{Ro2}
\|df\|^2\leq\dfrac{2d}{d+1} \dfrac{k}{K},
\end{align}
where $d=\max\limits_{p\in M} rk(df(p))$ is the maximal complex rank of $df$.
\end{thm}

In \cite{T1} Tosatti extends Yau's Schwarz lemma from a different direction. Namely, he considers almost Hermitian manifolds and proves the following.
\begin{thm}[Tosatti]\label{To2}
Let $(M, J, g)$ be a complete almost Hermitian manifold with second Ricci curvature bounded from below by $k\leq 0$, and with torsion and $(2,0)$ part of the curvature bounded. Let $(N, J_N, h)$ be an almost Hermitian manifold with holomorphic bisectional curvature bounded from above by $K
<0$. Then for any non-constant holomorphic map $f:M\rightarrow N$,  
\begin{align} \label{to3}
\|df\|^2\leq \dfrac{k}{K}.
\end{align}
\end{thm}

We prove the following result which generalizes the equality case of the Ahlfors-Schwarz lemma. 
\begin{thm} \label{them1} Assume that the equality holds at one point in $(\ref{Yau2}) or (\ref{Ro2}) or (\ref{to3})$. Then it holds on the whole manifold $M$ and $f$ is a totally geodesic map with constant rank. 
\end{thm}
The proof of the result shares some spirit of Minda in \cite{Mi}. Namely, by using the known bound of $\|df\|^2$ and analyzing the differential inequalities, we can invoke the classical strong maximum principle of linear elliptic operators to get the global equality. Then we obtain the geometric implications by the global equality (see also the appendix in \cite{Ni2}). Using the same technique, we also derive the result for the equality case for the volume version of Schwarz lemma. See Proposition \ref{prop3.1} in section 3. A special consequence of Theorem \ref{them1} is the following.
\begin{cor} \label{cor1.1}Let $S$ be a Riemann surface with a complete conformal metric $ds^2$ of Gaussian curvature at least $k<0$.  Let $(N,h)$ be a K\"ahler manifold with holomorphic sectional curvature bounded from above by $K<0$. For any non-constant holomorphic map $f:S\lra N$, if $f^*h=\dfrac{k}{K}ds^2$ at some point $p\in M$,  then $ds^2\equiv-\dfrac{1}{k}ds^2_H$ and $f$ is an isometric immersion (up to a scalar) onto the totally geodesic submanifold $f(S)$ in $N$.
\end{cor}
Here $ds^2_H$ is the unique hyperbolic metric on $S$. This corollary generalizes the equality case of Theorem 2 in \cite{RO2}. Theorem \ref{them1} may be generalized to the case of non-smooth functions and then to other forms of Schwarz lemmata in \cite{Ni2} . 

The second part of the paper is to derive a new Schwarz lemma on almost Hermitian manifolds, following the work of Ni \cite{Ni1} . Recently, by applying a viscosity consideration from PDE theory, Ni proves a Schwarz lemma on K\"ahler manifolds as follows (see also \cite{CCL} \cite{YC}): any holomorphic map from a complete K\"ahler manifold with holomorphic sectional curvature bounded from below by $K_1\leq 0$ into a K\"ahler manifold with holomorphic sectional curvature bounded from above by $K_2<0$ is distance decreasing up to a constant $\dfrac{K_1}{K_2}$, provided that the bisectional curvature of the domain is bounded from below. There are also several other types of Schwarz lemmata obtained in \cite{Ni1} and \cite{Ni2}. The aforementioned result looks rather symmetric as the distance decreasing constant only relies on the holomorphic sectional curvatures. Along this direction, we prove the following result on almost Hermitian manifolds.

\begin{thm}\footnote{After this work was completed, the authors noticed that Theorem \ref{T1.1} and Corollary \ref{cor1.2} were also proved by Weike Yu 	independently in \cite{Yu1}, using a different method.} \label{T1.1}
Let $(M, J, g)$ and $(N, J_N, h)$ be almost Hermitian manifolds with $M$ being complete. Suppose the holomorphic sectional curvatures of $M$ and $N$ satisfy $H_M\geq K_1$ and $H_N\leq K_2$, with $K_1\leq 0$ and $K_2<0$. Let $f: M\rightarrow N$ be a holomorphic map. Then $$f^*h \leq \dfrac{K_1}{K_2}g,$$ provided that the bisectional curvature of $M$ is bounded from below,  and the torsion and $(2,0)$ part of the curvature of $M$ are bounded. If $K_1=0$, then $f$ is a constant map.
\end{thm}

To prove the theorem, we apply the local holomorphic disc technique in \cite{RO1}. Namely, by choosing a suitable $J$-holomorphic disc in the maximum direction of $df$ and pulling back the metrics, we can roughly reduce the proof in higher dimension to the two-dimensional case, where the proof of Ahlfors-Schwarz lemma applies. To achieve the reduction, we compare the holomorphic sectional curvature of an almost Hermitian manifold and its almost complex submanifolds (see also \cite{Ko1}). As the almost complex structure may be non-integrable, we also need to show the existence of $J$-holomorphic discs with prescribed conditions. This is guaranteed by the work in \cite{IR}.

We remark that if $K_1>0$ and $K_2\leq 0$, the same proof of the above result shows that $f$ is also a constant map. When $M$ is compact, the torsion and the curvature tensor are automatically bounded. Then we get
\begin{cor} \label{cor1.2} Let $(M, J, g)$ and $(N, J_N, h)$ be almost Hermitian manifolds with $M$ being compact. Suppose the holomorphic sectional curvatures of $M$ and $N$ satisfy $H_M\geq K_1$ and $H_N\leq K_2$. For any holomorphic map $f: M\rightarrow N$,
\begin{itemize}
\item [(i)] if $K_1<0, K_2< 0$, then $f^*h \leq \dfrac{K_1}{K_2}g$;\\
\item [(ii)] if $K_1\geq 0, K_2<0$ or $K_1>0, K_2\leq 0$, then $f$ is a constant map.
\end{itemize}
\end{cor}
Part (ii) has been obtained previously by Yang \cite{Y1} on Hermitian manifolds by a different method. Later Masood \cite{M} generalizes it to almost Hermitian manifolds. Another corollary of Theorem $\ref{T1.1}$ is the following.
\begin{cor}\label{cor1.3}
Let $(M,J,g)$ be a compact almost Hermitian manifold with negative constant holomorphic sectional curvature. Then any biholomorphism of $(M,J)$ is an isometry of $(M,g)$. In particular, the group of biholomorphisms of $(M,J)$ is compact.
\end{cor}

When $(M,J)$ is a compact Hermitian manifold with negative holomorphic sectional curvature, Wu \cite{Wu} proves that the group of biholomorphism of $(M,J)$ is finite. To derive the finiteness in the above result, one still need to show that the group is discrete.

The structure of the paper is as follows. In section 2, we give the notations and background materials. In section 3, we prove Theorem \ref{them1}. In section 4, we discuss the holomorphic sectional curvature of the Chern connection and show the existence of some special $J$-holomorphic discs. In section 5, we prove Theorem \ref{T1.1} . Throughout the paper, we assume that $M$ is a connected manifold. $\|df\|^2$ is the trace function $tr_g(f^*h)$ and $\|df\|_0$ is the supreme norm of $df$.\\

\noindent \textbf{Acknowledgements.} We are very grateful to Professor Jiaping Wang for helpful suggestions and his support. We want to express our gratitude to Professors Kefeng Liu, Lei Ni and Fangyang Zheng for useful communications and discussions. We also thank Professor Jianfei Wang for helpful communications and the warm hospitality when we visited Huaqiao University in July 2021 and thank Professor Yibin Ren for some nice discussions.

 \section{Preliminaries}

In this section, we give some background materials, including almost Hermitian manifolds, the Chern connection and K\"ahler metrics, etc. We refer to \cite{T1}, \cite{TWY} and \cite{Z} for more details. 

Let $(M,J)$ be a $2m$-dimensional almost complex manifold and $g$ be a Riemannian metric on $M$. $(M, J, g)$ is called an almost Hermitian manifold if $g(JX,JY)=g(X,Y)$ for any $X,Y\in TM$. 
It is known that there is a unique connection $\nabla$ on $(M, J, g)$  (see e.g. \cite{TWY}) satisfying $$\nabla J=0,\ \ \nabla g=0$$ and $T(X,JX)=0$ for all $X\in TM$, where $$T(X,Y)=\nabla_XY-\nabla_YX-[X,Y]$$ is the torsion tensor. It is called the Chern connection or canonical connection. The last requirement is equivalent to that the (1,1) part of the torsion vanishes. The Nijenhuis tensor $N_J$ of $J$ is defined by  $$N_J(X,Y)=[X,Y]+J[JX,Y]+J[X,JY]-[JX,JY]$$ for $X,Y\in TM$. The K\"ahler form $\omega$ is defined by $\omega(X,Y)=g(JX,Y)$. An almost Hermitian manifold is a K\"ahler manifold if $N_J=0$ and $d\omega=0$. It is known that $(M, J, g)$ is K\"ahler if and only if $T=0$ \cite{Z}.

On an almost Hermitian manifold $(M, J, g)$, denote $T^{1,0}M$ and $T^{0,1}M$ the $\sqrt{-1}$ and $-\sqrt{-1}$ eigenbundles of $J$ in $TM\otimes_\mathbb R \mathbb C$. Then $TM\otimes_\mathbb R\mathbb C=T^{1,0}M\oplus T^{0,1}M$. For any $\xi\in T^{1,0}M$, we can write $\xi=\dfrac{1}{\sqrt{2}}(X-\sqrt{-1}JX)$ for some $X\in TM$. This gives an isomorphism between $ T^{1,0}M$ and $TM$. Given a smooth function $\phi$ on $M$, the complex Hessian of $\phi$ is defined to be $(\na^2\phi)^{(1,1)}$, i.e., the $(1,1)$ part of the Hessian $\na^2 \phi$. For instance, for $\xi,\eta\in T^{1,0}M$,
$$(\na^2\phi)^{(1,1)}(\xi, \bar{\eta})=\na^2\phi(\xi, \bar{\eta})=\xi(\bar{\eta}\phi)-(\na_{\xi}\bar{\eta})\phi.$$
Let $\{e_1, e_2, ..., e_m\}$ be a local unitary (1,0) frame on $M$ and $\{\theta^1,\theta^2, ...,  \theta^m \}$ be the dual frame so that $g=\sum_{i=1}^m\theta^i\otimes \bar{\theta}^i+\bar{\theta}^i\otimes\theta^i.$ 
Let $$\phi_{i\bar{j}}=\na^2\phi(e_i, \overline{e_j})=e_i\overline{e_j}\phi-(\nabla_{e_i}\overline{e_j})\phi.$$ As the $(1,1)$ part of the torsion of $\nabla$ vanishes, $\na_{e_i}\ov{e_j}-\na_{\ov{e_j}}e_i=[e_i, \ov{e_j}].$  Then $\phi_{i\bar{j}}=\phi_{\bar{j}i}$. Using Einstein summation convention, locally write
$$(\na^2\phi)^{(1,1)}=\phi_{i\bar{j}}(\theta^i\otimes \bar{\theta}^j+\bar{\theta}^j\otimes\theta^i).$$
The complex Laplacian of $\phi$ is defined to be \begin{align} \label{lap} \Delta \phi=tr_g(\na^2\phi)^{(1,1)}=2\sum_{i=1}^m\phi_{i\bar{i}}.\end{align} It is known that $\Delta$ is an elliptic operator which differs from the Laplacian of the Levi-Civita connection by some first order differential operators (Lemma 3.2 in \cite{T1}). Use $\langle\ , \ \rangle$ for the inner product $g$ and denote 
\begin{align*}
&R(X,Y)Z=\nabla_X\nabla_YZ-\nabla_Y\nabla_XZ-\nabla_{[X,Y]}Z,\\
&R(X, Y, Z, W)=\langle R(X,Y)Z, W\rangle,
\end{align*}
to be the curvature tensors of $\nabla$.  
\begin{defn}
Given two nonzero $(1,0)$ vectors $\xi=\frac{1}{\sqrt{2}}(X-\sqrt{-1}JX), \eta=\frac{1}{\sqrt{2}}(Y-\sqrt{-1}JY)$, 
 the holomorphic bisectional curvature of $\nabla$ in directions $\xi, \eta$ is defined to be $$B(\xi,\eta)=\dfrac{R(\xi,\bar{\xi},\eta,\bar{\eta})}{|\xi|^2|\eta|^2}=\dfrac{R(X,JX,JY,Y)}{|X|^2|Y|^2},$$
where $|\xi|^2=|X|^2=\langle X, X\rangle$ and $|\eta|^2=|Y|^2=\langle Y,Y\rangle$. The holomorphic sectional curvature in the direction $\xi$ or $X$ is defined to be
 $$H(\xi)=B(\xi,\xi)=R(X,JX,JX,X)/|X|^4.$$
 We also use $H(X)$ for $H(\xi)$ in the context. The first and second Ricci curvatures of $(J,g)$ are defined by $$Ric^{(1)}(\xi,\bar{\eta})=\sum_{i=1}^m R(\xi,\bar{\eta},e_i,\ov{e_i}), \ \ Ric^{(2)}(\xi,\bar{\eta})=\sum_{i=1}^m R(e_i,\ov{e_i}, \xi,\bar{\eta}).$$ The Chern scalar curvature is defined to be $R=\sum_{i,j=1}^m R(e_i,\ov{e_i},e_j,\ov{e_j})$.
\end{defn}
We remark that if $(M,J,g)$ is K\"ahler, then $Ric^{(1)}=Ric^{(2)}$ by the K\"ahler symmetry. It is called the Ricci curvature in short. We have the following notions of boundedness on an almost Hermitian manifold (cf. \cite{T1}).
\begin{defn}
\begin{itemize}
\item[(i)] The first (resp. second) Ricci curvature is said to be bounded from below if there is a constant $C$ such that $Ric^{(1)}(\xi,\bar{\xi})\geq C|\xi|^2$ (resp. $Ric^{(2)}(\xi,\bar{\xi})\geq C|\xi|^2$) for any $\xi\in T^{1,0}M$. \vspace{.1cm}
\item[(ii)] The holomorphic bisectional curvature of $M$ is said to be bounded from below if there is a constant $C_1$ such that $B(\xi,\eta)\geq C_1$ for all $\xi, \eta\in T^{1,0}M$. \vspace{.1cm}
\item[(iii)] The torsion tensor $T$ is said to be bounded if there is a constant $C_2>0$ such that $|T(\xi,\eta)|\leq C_2|\xi||\eta|$ for any $\xi,\eta\in T^{1,0}M$. \vspace{.1cm}
\item[(iv)] The $(2,0)$ part of the curvature $R$ is said to be bounded if there is a constant $C_3>0$ such that $|R(\xi,\eta,\xi, \bar{\eta})|\leq C_3|\xi|^2|\eta|^2$ for any $\xi, \eta\in T^{1,0}M$.
\end{itemize}
\end{defn}
Note that the $(2,0)$ part of the curvature tensor vanishes when $M$ is Hermitian, i.e. $N_J=0$. If $M$ is compact, then it clearly satisfies all the boundedness conditions. For a fixed point $o\in M$, denote $r(q)=d(o,q)$ the distance from $o$ to $q\in M$. It is smooth away from $o$ and its cut locus. The following complex Hessian comparison theorem on almost Hermitian manifolds due to Tosatti \cite{T1} will be used in the proof of Theorem \ref{T1.1} (also see \cite{Yu}). 
\begin{thm}[Tosatti] \label{to}
Let $(M,J,g)$ be an almost Hermitian manifold with bisectional curvature bounded from below by $C_1$, torsion bounded by $C_2$ and $(2,0)$ part of the curvature bounded by $C_3$. Then at any point where $r$ is smooth, $$r_{i\bar{j}}\leq (\dfrac{1}{r}+C)g_{i\bar{j}},$$
where $C$ depends on $C_1,C_2,C_3$ and the dimension of $M$. \end{thm}
Here $A_{i\bar{j}}\leq B_{i\bar{j}}$ means that $(B_{i\bar{j}}-A_{i\bar{j}})$ is positive semi-definite. The proof of the above theorem is contained in \cite{T1} during the proof of the Laplacian comparison theorem. 
Explicitly, it follows from the expressions (4.30), (4.31), (4.32) and the argument after (4.32) there. As $r_{i\bar{j}}=\dfrac{1}{2}\sqrt{-1}d(Jdr)(e_i,\overline{e_j})$,  the factor $2$ does not appear in the complex Hessian comparison.

 In Proposition 2 in \cite{RO1}, an existence result of supporting functions is stated. As in non-K\"ahler setting, it relies on the complex Hessian comparison, it seems that the condition of Riemmanian sectional curvature bounded from below there should be replaced by conditions on bisectional curvature and torsion or other related conditions.

Let $(N,J_N,h)$ be a $2n$-dimensional almost Hermitian manifold.
\begin{defn}
A smooth map $f:M\lra N$ is called a $(J,J_N)$ holomorphic map (or holomorphic map in short) if it satisfies $$df\circ J=J_N\circ df.$$
\end{defn}
Let $T^{1,0}M$ and $T^{1,0}N$ be the bundle of $(1,0)$ tangent vectors of $M$ and $N$. By definition, $df(T^{1,0}M)\subset T^{1,0}N$. For any $p\in M$, choose a local unitary frame $\{\eta_1, \eta_2, ..., \eta_n\}$ near $q=f(p)\in N$ such that $df$ is given by $df(e_i)=f_i^\alpha\eta_\alpha$. The norm of $df$ is defined to be $$\|df\|^2=\sum_{i=1}^{m}\sum_{\al=1}^{n}|f_i^{\al} |^2.$$
It is easy to show that $\|df\|^2$ is independent of the chosen frames and is globally defined. It is also equal to the trace function $tr_g(f^*h)$.

\section{The equality case of Schwarz lemmata}
In this section, we study the equality case of the Schwarz lemmata of Yau, Royden and Tosatti. First, the following strong maximum principle of Hopf is known (cf. Theorem 3.5 in \cite{GT} for example).
\begin{prop}\label{Hopf} Let $L$ be a uniformly elliptic linear differential operator of the form 
$$Lu=a^{ij}D_{ij}u+b^i(x)D_i u+c(x)u, \ \  a^{ij}=a^{ji}$$ in a domain $\Omega\subset \mathbb{R}^n$. Let $\lambda(x)$ be the minimum eigenvalues of $(a^{ij})$.  Suppose that $c\leq 0$ and $\dfrac{c}{\lambda}$ is bounded. If $u\in C^2(\Omega)$ satisfies $Lu\geq 0$, then $u$ cannot achieve a non-negative maximum in the interior of $\Omega$ unless it is constant.
\end{prop}
We then apply the differential inequalities in \cite{RO1}\cite{T1} to prove Theorem \ref{them1}.
\begin{proof} [Proof of Theorem 1.4] Assume that we are in the setting of Theorem \ref{Ro1} and let $v=\|df\|^2$. Suppose at some $p\in M$, $v(p)=\frac{2d}{d+1} \frac{k}{K}>0$ with $k<0, K<0$ (the case of $k=0$ or $d=0$ is trivial). By Proposition 4 in \cite{RO1}, at any $q\in M$,
\begin{align}
\Delta \log v\geq k-\dfrac{rk(q)+1}{2\ rk(q)} Kv\geq k-\dfrac{d+1}{2d} Kv,
\end{align} where $rk(q)$ is the complex rank of $df$ at $q\in M$ and $\Delta$ is the complex Laplacian defined in $(\ref{lap})$. By (\ref{Ro2}), the following holds on $M$,
\begin{align*}
0\leq \dfrac{K}{2k}\dfrac{d+1}{d}v \leq 1.
\end{align*}
Also, in a small neighbourhood $U$ of $p$, $v>0$.  Let $\dfrac{K}{2k}\dfrac{d+1}{d}=C_1$.  As $0<C_1v\leq 1$ in $U$, the elementary inequality
$1-C_1v\leq-\log (C_1v)$ holds. Denote $u=\log (C_1v)\leq 0$. Then we have 
\begin{align*}
\Delta \log v&\geq k-\dfrac{d+1}{2d} Kv\\
&\geq -k\log (C_1v), 
\end{align*} 
on $U$, which gives that 
\begin{align*}
\Delta u+k u\geq 0.
\end{align*}
The complex Laplacian $\Delta$ naturally satisfies the conditions in Proposition \ref{Hopf} in $U$. As $u\leq 0$ with $u(p)=0$ and  $ k<0$, by Proposition \ref{Hopf}, $u$ must be constant zero and $v=\frac{2d}{d+1} \frac{k}{K}$ in $U$.  It follows that the set $\{q\in M|v(q)=\frac{2d}{d+1} \frac{k}{K}\}$ is open and closed. By the connectedness of $M$, $v$ is constant on $M$ and $rk(q)=d$ for any $q\in M$.

For the inequalities (\ref{Yau2}),(\ref{to3}), as Theorem \ref{To2}  extends Theorem \ref{Yau}, we only need to discuss the case of (\ref{to3}). Assume that we are in the setting of Theorem \ref{To2} and $K<0$. Let $v=\|df\|^2$. By (5.41) and Theorem 4.1 in \cite{T1},  $\Delta v\geq 2(kv-Kv^2),$ and $v\leq \dfrac{k}{K}$.  Then 
\begin{align*}
\Delta v&\geq 2Kv\left(\dfrac{k}{K}-v\right)\\
&\geq 2k\left(\dfrac{k}{K}-v\right).
\end{align*}
Let $u=v-\dfrac{k}{K}$, then $\Delta u\geq -2ku$. Note that $u\leq 0$. The same argument as the above shows that if the equality of (\ref{to3}) holds at some $p\in M$, then $u=0$ in a neighborhood of $p$. It follows that $v=\dfrac{k}{K}$ in a neighborhood of $p$. By the connectedness of $M$, $v$ is constant on the whole manifold.

Once we have the global equality of $v$, by the formulas in the proof of Proposition 4 in \cite{RO1} (see also inequality (A.3) in \cite{Ni1}) and (5.41) in \cite{T1}, then $\nabla df=0$. So $f$ is totally geodesic and the rank of $df$ is constant in all cases \cite{V}.  \end{proof}

\begin{remk} \label{remk3.1}
From the above proof and Proposition \ref{Hopf}, we can obtain a special strong maximum principle as follows.
Let $L$ be a uniformly elliptic linear differential operator of the form 
$$Lu=a^{ij}D_{ij}u+b^i(x)D_i u, \ \  a^{ij}=a^{ji}$$ in a domain $\Omega\subset \mathbb{R}^n$. Suppose $u\leq C$ and  $Lu\geq f(u)$ for some function $f$ with  $f(C)\geq 0$ and  $f'(x)\leq M, \forall\ x\leq C$ for some $M>0$. Then $u$ can not equal $C$ in the interior of $\Omega$ unless it is constant.
\end{remk}
In \cite{Ya1}, Yau also obtained the volume form version of the Schwarz lemma for holomorphic maps under weaker curvature conditions. It was extended by Tosatti \cite{T1} to almost Hermitian manifolds as follows.
\begin{thm}[Tosatti]\label{T5}
Let $(M, J, g)$ be a $2m$-dimensional complete almost Hermitian manifold with Chern scalar curvature bounded from below by $mK_1<0$. Let $(N, J_N, h)$ be another $2m$-dimensional almost Hermitian manifold with the first Ricci curvature bounded from above by $K_2<0$. If the second Ricci curvature of $(M,g)$ is bounded from below and with torsion and $(2,0)$ part of the curvature bounded, then for any non-degenerate holomorphic map $f:M\rightarrow N$,  
\begin{align} \label{to4}
f^*dV_h\leq (\dfrac{K_1}{K_2})^m dV_g.
\end{align}
\end{thm}
Here $dV_g$ and $dV_h$ are the volume forms on $(M, g)$ and  $(N, h)$ respectively and the map $f$ is non-degenerate if $f^*dV_h$ is nowhere vanishing. Applying Remark \ref{remk3.1}, we prove the following result. We would like to thank Professor Kefeng Liu for pointing it out to us.
\begin{prop}\label{prop3.1}
Assume that the equality holds at one point in $(\ref{to4})$, then it holds on the whole manifold $M$.
\end{prop} 
\begin{proof}
Let $v=\dfrac{\det f^*h}{\det g}$. Then $v\leq (\dfrac{K_1}{K_2})^m$ by Theorem \ref{T5}. By the Chern-Lu formula for the volume form (see Page 1081 in \cite{T1}), we have
$$\Delta v\geq -2mK_2v^{1+\frac{1}{m}}+2mK_1v.$$ Let $f(v)=-2mK_2v^{1+\frac{1}{m}}+2mK_1v$. Then $f'(v)\leq -2K_1$ as $v\leq (\dfrac{K_1}{K_2})^m$ . The result then follows from Remark \ref{remk3.1}.
\end{proof}

\begin{proof}[Proof of Corollary \ref{cor1.1}]  By Theorem \ref{them1}, $tr_gf^*h\equiv \frac{k}{K}$ and $f$ is totally geodesic. Then the image $f(S)$ is a totally geodesic submanifold of $N$. As $f$ is non-constant and $m=1$, then $f$ is an immersion and $f^*h=\frac{k}{K}ds^2$. By the proof of Proposition 4 in \cite{RO1}, the equality forces that the Gaussian curvature of $ds^2$ equals $k$. Therefore, $ds^2=-\dfrac{1}{k}ds^2_H$ by the uniqueness of the hyperbolic metric on $S$.
\end{proof}

\begin{remk} \label{remk3.2} Recently in \cite{KB}, Broder extends Schwarz type inequalities to more general settings on Hermitian manifolds. He informed us that our proof may treat the equality case of the inequalities in his paper. Indeed, the differential inequalities derived there satisfy the conditions of Remark \ref{remk3.1} and the analogous equality result follows (c.f. Remark 1.10 in \cite{KB}). We would like to thank him for the communications.
\end{remk}

\section{Holomorphic sectional curvature and $J$-holomorphic disc}
In this section, we compare the holomorphic sectional curvatures of an almost Hermitian manifold and its almost complex submanifolds. Then we discuss the existence of some special $J$-holomorphic discs through a point in any direction on an almost Hermitian manifold, based on the work of \cite{IR}.
\begin{defn} Let $(S, i)$ be an immersed submanifold of $(M, J)$ with $i:S\lra M$ being the immersion. For any $p\in S$, identify $T_pS$ with $i_*(T_pS)$ in $T_{i(p)}M$. If $J(T_pS)\subset T_pS$, then we say $(S, i)$ is an almost complex submanifold of $(M, J)$.
\end{defn}
The pull back of $(J,g)$ through $i$ induces an almost Hermitian structure on $S$ which we still denote by $(J,g) $. As an immersion is locally embedding, to focus on local properties, from now on we assume that $S$ is an embedded almost complex submanifold and identify $S$ with $i(S)$. 

Denote $\nabla $ and $\hat{\nabla}$ the Chern connections on $M$ and $S$ respectively. Let $NS$ be the normal bundle of $S$ and $\pi^\top$ and $\pi^\bot$ be the tangential and normal projections:
\begin{align*}\pi^\top: \  TM|_S\rightarrow TS,\ \ \ 
\pi^\bot: \ TM|_S\rightarrow NS.
\end{align*} Then $\nabla $ and $\hat{\nabla}$ are related as follows.
\begin{prop} \label{p3.2}Let $X, Y$ be in $\Gamma(TS)$ which are extended to vector fields on $M$. Then we have
$$\hat{\nabla}_XY=(\nabla _XY)^\top.$$
\end{prop}
\begin{proof} Define $\nabla^\top_XY:=(\nabla _XY)^\top$. It is clear that $\nabla^\top$ is a connection on $S$. By the uniqueness of Chern connection,  to show that $\nabla^\top=\hat{\nabla}$,  it suffices to show that $\nabla^\top g=0, \ \nabla^\top J=0$ and its torsion has vanishing $(1,1)$ part. Let $X, Y, Z\in \Gamma(TS)$ which are extended to vector fields on $M$. At any $p\in S$, we have
\begin{align*}
X\langle Y, Z\rangle=&\langle\nabla _X Y, Z  \rangle+ \langle Y,  \nabla _ XZ\rangle\\
=&\langle(\nabla _X Y)^\top, Z  \rangle+ \langle Y, (\nabla _ XZ)^\top\rangle\\
=&\langle\nabla^\top_X Y, Z  \rangle+ \langle Y,\nabla^\top_XZ\rangle
\end{align*}
Thus $\nabla^\top g=0$. Also, as $JX\in T_pS$, for any $W\in \Gamma(NS)$, 
$\langle X, JW\rangle=-\langle JX, W\rangle=0.$ So $JW\in \Gamma(NM)$. Then
\begin{align*}
\nabla_X^\top(JY)&=(\nabla _X (JY))^\top=(J\nabla _X Y)^\top\\
&=(J\nabla^\top_XY+J(\nabla _XY)^\bot)^\top\\
&=J\nabla^\top_XY.
\end{align*}
So $\nabla^{\top}J=0$.
Finally, let  $X, Y$ be (1,0) and (0,1) vector fields of $S$ respectively.  Then
\begin{align*}
\nabla^\top_ XY-\nabla_Y^\top X&=(\nabla _XY-\nabla _YX)^\top\\
&=[X, Y]^\top=[X, Y]. 
\end{align*}
So the torsion of  $\nabla^\top$ has vanishing (1,1) part. Therefore, $\nabla^\top=\hat{\nabla}$.
\end{proof} 
For any $X,Y\in TS$, the second fundamental form of $S$ is given by $$\textit{II}\ (X, Y):= (\nabla _X Y)^\bot. $$ By Proposition \ref{p3.2} and the metric compatibility of $\nabla $, we have 
\begin{prop}  \label{p3.3}Let $X, Y\in \Gamma (TS)$ and $W\in \Gamma (NS)$ which are extended to vector fields on $M$. Then 
\begin{align*}\nabla _XY&=\hat{\nabla}_XY+\mathit{II} (X, Y)\\
\langle \nabla _XW, Y\rangle &=-\langle W, \mathit{II} (X, Y)\rangle. \end{align*}
\end{prop}
Denote $R$ and $\hat{R}$ the curvature for $\nabla $ and $\hat{\nabla}$. With Proposition \ref{p3.3}, we get the Gauss type equation by direct calculation.
\begin{prop}For any $X, Y, Z, W\in T_pS$, \label{p3.4}
$$R(X, Y, Z, W)=\hat{R}(X, Y, Z, W)- \langle \mathit{II}(X, W), \mathit{II} (Y,Z)\rangle+ \langle\mathit{II}(X,Z), \textit{II}(Y,W)\rangle.$$
\end{prop}
We are then able to compare the holomorphic sectional curvature of $\nabla$ and $\hat{\nabla}$. 
\begin{cor}\label{c3.5}
$$H(X)=\hat{H}(X)+2|\mathit{II}(X, X)|^2/|X|^4,$$
for $X\in TS$. Hence, $\hat{H}(X)\leq H(X)$ with equality holding if and only if $\mathit{II} (X,X)=0$. 
\end{cor}
The above inequality was obtained previously in \cite{Ko1} by using local frames and structure equations.
\begin{proof} By $\nabla J=0$, we get $\nabla _X (JY)=J(\nabla _X Y)$. So $\mathit{II} (X, JY)=J\mathit{II} (X, Y)$.
As $T(X, JX)=0$ gives $$\na_{JX}X-\na_{X}(JX)=[JX, X],$$
and $JX\in TS$, then
 \begin{align*}
 \textit{II}(JX, X)=(\na_{JX}X)^\bot=(\na_X(JX)+[JX, X])^\bot=\textit{II} (X,JX).
 \end{align*}
 By Proposition \ref{p3.4}, we have 
 \begin{align*}H(X)&=R(X,JX,JX,X)/|X|^4\\
&= \hat{H}(X)+ \dfrac{1}{|X|^4}(\langle \mathit{II}(X, X), \mathit{II} (X,X)\rangle+ \langle\mathit{II}(X,JX), \textit{II}(JX,X)\rangle)\\
 &=\hat{H}(X)+ 2|\mathit{II}(X, X)|^2/|X|^4\end{align*}
 \end{proof}
 Next, we apply Corollary \ref{c3.5} to study the curvature of the pull back metric on a $J$-holomorphic disc on $M$. Denote $\mathbb D=\{z=x+\sqrt{-1}y\in \mathbb C| |z|<1\}$ the unit disc associated with the standard complex structure $J_{st}$.
\begin{defn}
 A $J$-holomorphic disc on $M$ is a smooth map $\varphi:\mathbb D \lra M$ such that $J\circ d\varphi=d\varphi\circ J_{st}$.\end{defn}
By the work of Nijenhuis-Woolf, given any point $p\in M$ and $(1,0)$ vector $\xi$ in $T^{1,0}_pM$ with $|\xi|$ small enough, there exists a $J$-holomorphic disc on $M$ such that $\varphi(0)=p$ and $\varphi_*(\frac{\pa}{\pa z})(p)=\xi$. Equivalently, for a sufficiently small real vector $v\in T_pM$, there exists a $J$-holomorphic disc with $\varphi(0)=p$ and $\dfrac{\partial \varphi}{\partial x}(0)=\varphi_*(\frac{\partial}{\partial x} )(p)=v$. The following more general existence result is proved in \cite{IR} (Proposition 1.1) using the implicit function theorem.
\begin{thm} [Ivashkovich-Rosay] \label{T3.7}
Let $k\in \mathbb N$, $k\geq 1$. Let $J$ be a smooth almost complex structure near $0$ in $\mathbb R^{2n}$. For any $p\in \mathbb R^{2n}$ close enough to 0, and every $\{v_1,\cdots,v_k\} \in (\mathbb R^{2n})^k$ small enough, there exist a smooth $J$-holomorphic disc $\varphi:\mathbb D\lra \mathbb R^{2n}$ such that $\varphi(0)=p$ and $\dfrac{\partial ^l\varphi}{\partial x^l}(0)=v_l$,  for any $1\leq l\leq k$.
\end{thm}
Then we can prove
\begin{prop} \label{p3.8} Let $(M,J,g)$ be an almost Hermitian manifold. For any $p\in M$ and $v\in T_pM$ with $|v|$ small enough, there exists a $J$-holomorphic disc $\varphi:\mathbb D\lra M$ such that $\varphi(0)=p, \dfrac{\partial \varphi}{\partial x}(0)=v$ and the Gaussian curvature of $\varphi^*g$ at $0$ identifies with the holomorphic sectional curvature $H(v)$.
\end{prop}
\begin{proof}
Let $t^1, t^2, ..., t^{2m}$ be local coordinates of $M$ at $p$.  Let $\Gamma_{ij}^k$ be the real Christoffel symbols of the Chern connection $\nabla$. Write $v=a^i \dfrac{\pa}{\pa t^i}$. For $|v|$ small enough, by Theorem \ref{T3.7},  there exists a $J$-holomorphic disc $\varphi:\mathbb D\lra M$ such that $\varphi(0)=p$, $\dfrac{\partial \varphi}{\partial x}(0)=v$
and 
$$\dfrac{\pa^2\vp^k}{\pa x^2}(0)+a^ia^j\Gamma_{ij}^k(p)=0,$$
where $\varphi=(\varphi^1,\cdots, \varphi^{2n}).$ If $v$ is nonzero, then $\varphi$ is an embedding in a neighborhood $U$ of $0$. Denote $X=\varphi_*(\dfrac{\pa}{\pa x})$ the vector field near $p$, which extends to a local vector field on $M$. Then 
\begin{align*} 
\nabla_XX(p)=\nabla_{\vp_*(\frac{\pa}{\pa x})}(\dfrac{\pa\vp^i}{\pa x}\dfrac{\pa}{\pa t^i})(p)
=(\dfrac{\pa^2\vp^k}{\pa x^2}(0)+a^ia^j\Gamma_{ij}^k(p))\dfrac{\pa}{\pa t^k}=0
\end{align*}
Therefore, $\mathit{II} (X,X)(p)=(\nabla_XX)^\bot(p)=0$. Denote $K(z)$ the Gaussian curvature of $\varphi^*g$ in $U\subset \mathbb D$, which is equal to the holomorphic sectional curvature of $(\varphi(U)$. Therefore, by Proposition \ref{c3.5}, $K(0)=H(v)$.
\end{proof}

\section{A Schwarz lemma on almost Hermitian manifolds}
In this section, we prove Theorem \ref{T1.1}. Let $(M,J,g)$ and $(N,J_N,h)$ be two almost Hermitian manifolds and $f:M\lra N$ a $(J,J_N)$ holomorphic map. Denote $H_M$ and $H_N$ the holomorphic sectional curvatures of $M$ and $N$ and assume $H_M\geq K_1, H_N\leq K_2$ for some constants $K_1\leq 0, K_2<0$. Define the supreme norm of $df$ as $\|df\|_0$ by
\begin{align*}
\|df\|_0(p)=\sup_{\xi\neq 0}\dfrac{|df(\xi)|_h}{|\xi|_g},\ \ \ p\in M, \ \xi\in T_p^{1,0}M.\end{align*}
Our main object is to estimate $||df||_0$ under the conditions in Theorem \ref{T1.1}, since $||df||^2_0\leq \dfrac{K_1}{K_2}$ implies that $f^*h\leq \dfrac{K_1}{K_2} g$.
First by definition, for any $p\in M$ and any $\xi\in T_p^{1,0}M$, we have $|df(\xi)|_h\leq \|df\|_0(p)|\xi|_g$. By change of unitary frames, we may assume that there are local unitary (1,0) frames $\{e_1, e_2, ..., e_m\}$ near $p\in M$ and $\{\eta_1, \eta_2, ..., \eta_n\}$ near $q=f(p)\in N$ such that $df$ is given by $df(e_i)=f_i^\alpha\eta_\alpha$ with $f_i^\alpha= \lambda_i\delta_i^\alpha$ and $\lambda_1\geq \lambda_2\geq...\geq \lambda_r> \lambda_{r+1}=...=0$, and
$r$ being the rank of $(f_i^\alpha)$. Then $\|df\|_0=\lambda_1$ and $|df(e_1)|_h=\|df\|_0|e_1|_g$. We estimate $\lambda_1$ in the following.

Let $\xi_1=t e_1\in T_p^{1,0}M$ with $t>0$ small enough. By Proposition \ref{p3.8}, there is a $J$-holomorphic disc $\varphi:\mathbb D\rightarrow M$ with $\varphi(0)=p,\  \dfrac{\pa\varphi}{\pa x}(0)=Re(\xi_1)$ and the Gaussian curvature of $\varphi^*g$ at $0$ is $H_M(\xi_1)$.  Denote $$\varphi^*g=u^2dzd\bar{z}$$ for some positive function $u$ near $0$.
 As $H_M\geq K_1$, and the Gaussian curvature for $\varphi^*g$ is $- \dfrac{4}{u ^2}\dfrac{{\pa}^2\log u }{\pa z\pa\bar{z}} $, we have 
 \begin{align*}
- \dfrac{4}{u ^2}\dfrac{{\pa}^2\log u }{\pa z\pa\bar{z}}(0)\geq K_1.
 \end{align*}
Since both $\varphi$ and $f$ are holomorphic, if $\lambda_1(p)\neq 0$, there is a neighbourhood $U$ of $0$ in $\mathbb D$ such that$(U, f\circ \varphi)$ is an almost complex submanifold of $(N,J_N)$. Then we pull back $h$ to $U$ to get
\begin{align*}
\vp^*(f^*h)= v^2dzd\bar{z}.
\end{align*}
We have
\begin{align}
v(z)=|f_*\vp_*(\frac{\pa}{\pa z})|_h\leq \|df\|_0 |\vp_*(\frac{\pa}{\pa z})|_g=\lambda_1(\varphi(z)) u(z). \label{e4.1}
\end{align}
As $H_N\leq K_2$, by Corollary \ref{c3.5}, we get $$- \dfrac{4}{v ^2}\dfrac{{\pa}^2\log v }{\pa z\pa\bar{z}}(0)\leq K_2.$$
 Now let $\phi$ be an arbitrary smooth function on $M$. Recall that the complex Hessian of $\phi$ is 
$$(\na^2\phi)^{(1,1)}=\phi_{i\bar{j}}\theta^i\otimes \bar{\theta}^j+\phi_{\bar{i}j}\bar{\theta}^i\otimes\theta^j,$$
where $$\phi_{i\bar{j}}=\na^2\phi(e_i,\bar{e}_j)=e_i\bar{e}_j\phi-\na_{e_i}\bar{e}_j\phi=\phi_{\bar{j}i}.$$
We pull back $\phi$ to $\mathbb D$ by $\varphi$ and compute its differentials.
As $\vp$ is holomorphic, we may write $\vp_*(\frac{\pa}{\pa z})
=\vp^{i}_{ z}e_{i}$ and then $\vp_*(\frac{\pa}{\pa\bar{ z}})=\overline{\vp}^{i}_{ z}\bar{e}_{i}$, where $\overline{\vp}^{i}_{ z}$ is the complex conjugate of $\vp^{i}_{ z}$.
Then
\begin{align}
\dfrac{\pa \phi}{\pa\bar{ z}}=&\vp_*(\frac{\pa}{\pa\bar{ z}})\phi=\overline{\vp}^{j}_{ z}\bar{e}_j\phi \notag\\
\dfrac{\pa^2 \phi}{\pa z\pa\bar{ z}}=&\vp_*(\frac{\pa}{\pa z})(\bar{\vp}^{j}_{ z}\overline{e}_j \phi) \notag \\
=&\vp^{i}_{ z}\bar{\vp}^{j}_{ z}e_{i} e_{\bar{j}}\phi+\vp^{i}_{ z}(e_{i} \bar{\vp}^{j}_{ z})\overline{e}_j \phi. \label{4.1}
\end{align}
Since $[\vp_*\frac{\pa}{\pa z}, \vp_*\frac{\pa}{\pa\bar{ z}}]
=\vp_*[\frac{\pa}{\pa z}, \frac{\pa}{\pa\bar{ z}}]=0,$
we have 
\begin{align*}
\vp^{i}_{ z}(e_{i} \overline{\vp}^{j}_z)\overline{e}_{j}
+\vp^{i}_{ z} \overline{\vp}^{j}_z[e_{i}, \overline{e}_{j}]
- \overline{\vp}^{j}_z(\overline{e}_{j}\vp^{i}_{ z})e_{i}=0.
\end{align*}
As $
\na_{e_i}\bar{e}_j -\na_{\bar{e}_j} e_i=[e_i, \bar{e}_j]$, then $
\na_{e_i}\bar{e}_j =[e_i, \bar{e}_j]^{(0,1)}$.
It follows that 
\begin{align*}
\vp^{i}_{ z} \overline{\vp}^{j}_z\na_{e_i}\bar{e}_j=&\vp^{i}_{ z} \overline{\vp}^{j}_z[e_{i}, \overline{e}_{j}]^{(0,1)}=-\vp^{i}_{ z}(e_{i} \overline{\vp}^{j}_z)\overline{e}_{j}
\end{align*}
Putting it into (\ref{4.1}), we get \begin{align} \label{5.3} \dfrac{\pa^2 \phi}{\pa z\pa\bar{ z}}=\vp^{i}_{ z}\bar{\vp}^{j}_{ z}\phi_{i\bar{j}}.\end{align}
Now we finish the proof of Theorem \ref{T1.1}.

\begin{proof}[Proof of Theorem \ref{T1.1}]
Fix a point $o\in M$. Let $r(q)=d(o, q)$ be the distance function on $M$. For $a>0$, denote $$D_a=\{q\in M \  |\  r(q)<a\}, \ \ \ \ \overline{D}_a=\{q\in M \  |\  r(q)\leq a\}.$$ Define a function  $\rho$ on $D_a$ as 
$$\rho=(1-\dfrac{r^2}{a^2})\|df\|_0.$$
Then $\rho\geq 0$ on $\overline{D}_a$ and achieves its maximum at some $p\in D_a$. If $p$ is not a cut point of $o$, then $r$ is smooth at $p$. Otherwise, we can use the Calabi's trick to deal with it. Indeed, let $\gamma$ be a minimal unit speed geodesic from $o$ to $p$. Take a point $q_0$ on $\gamma$ which is not a cut point of $p$ and denote $b=d(o, q_0)$. For any $q\in M$, $\wtl{r}(q)=b+d(q_0, q)$ is smooth near $p$. Note that $\wtl{r}\geq r$ and $\wtl{r}(p)=r(p)$. Choose $\delta$ small enough such that the open geodesic ball $B_p(\delta)\subset D_a$ and $\wtl{r}(q)<a$ is smooth in $ B_p(\delta)$. It follows that $\wtl{\rho}=(1-\dfrac{\wtl{r}^2}{a^2})||df||_0\leq \rho$ and achieves the same maximum at $p$. Then we may use $\wtl{\rho}$ to replace $\rho$ and apply the maximum principle to $\wtl{\rho}$. So we just assume that $r$ is smooth at $p$.

Assume that $\|df\|_0(p)=\la_1(p)>0$ (otherwise, there is nothing to prove). Here $\lambda_1$ is the greatest eigenvalue of $df$ introduced previously. Let $\xi_1$ be an eigenvector of $\lambda_1$ at $p$ with $|\xi_1|_g$ small enough. By the previous argument, there is a $J$-holomorphic disc $\varphi:\mathbb D\rightarrow M$ such that $\varphi(0)=p,\  \dfrac{\pa\varphi}{\pa x}(0)=Re(\xi_1)$ and the Gaussian curvature of $\varphi^*g$ at $0$ is $H_M(\xi_1)$. Consider the function \begin{align}
\hat{\rho}(z)=\left(1-\dfrac{r^2(\vp(z))}{a^2}\right)\dfrac{v(z)}{u(z)}
\end{align} 
on $\mathbb D$. By (\ref{e4.1}), $v(z)\leq \lambda_1(\varphi(z))u(z)$ and $v(0)=\lambda_1(p)u(0)$. So $\hat{\rho}(z)\leq \rho(\varphi(z))$ and $\hat{\rho}(0)=\rho(\varphi(0))$. Therefore, $\hat{\rho}$
achieves its maximum at $z=0$. We have at $z=0$,
\begin{align*}
0\geq &\frac{\pa^2}{\pa z\pa\bar{ z}}\log \hat{\rho}\\
=&-2\frac{a^2+r^2}{(a^2-r^2)^2}\left|\frac{\pa r}{\pa \bar{z}}\right|^2-\frac{2r}{a^2-r^2}\frac{\pa^2 r}{\pa z\pa\bar{ z}}+\frac{\pa^2}{\pa z\pa\bar{ z}}\log\frac{v(z)}{u(z)}.
\end{align*}
Denote $r_i=e_ir, r_{\bar{i}}=\bar{e}_ir$. As $1=|\na r|^2_g=\sum_{i=1}^m|r_i|^2$, we have
\begin{align*}
\left|\frac{\pa r}{\pa\bar{ z}}\right|^2=|\sum_{i=1}^m\overline{\vp}^{i}_{ z}r_i|^2\leq (\sum_{i=1}^m|\vp^i_z|^2)(\sum_{i=1}^m|r_i|^2)=u^2(z).
\end{align*}
By (\ref{5.3}) and Theorem \ref{to},
\begin{align*}
\frac{\pa^2 r}{\pa z\pa\bar{ z}}=\vp^{i}_{ z}\bar{\vp}^{j}_{ z}r_{i\bar{j}}
\leq (\frac{1}{r}+C)g_{i\bar{j}}\vp^{i}_{ z}\bar{\vp}^{j}_{ z}
=(\frac{1}{r}+C)u^2(z).
\end{align*}
Therefore, at $z=0$,
\begin{align*}
0\geq -2\frac{a^2+r^2}{(a^2-r^2)^2}u^2(0)-(\frac{1}{r}+C)\frac{2r}{a^2-r^2}u^2(0)
-\frac{K_2}{4}v^2(0)+\frac{K_1}{4}u^2(0).
\end{align*}
As $K_2< 0$, 
 \begin{align*}
 \la_1^2(p)=&\dfrac{v^2(0)}{u^2(0)}\leq \dfrac{K_1}{K_2}-\frac{8(a^2+r^2)}{K_2(a^2-r^2)^2}-\frac{8(1+Cr)}{K_2(a^2-r^2)}.
 \end{align*}
 Thus for any $q\in D_a$,
 \begin{align*}
\rho^2(q)\leq& \left(\dfrac{a^2-r^2(p)}{a^2}\right)^2\left(\dfrac{K_1}{K_2}-\frac{16a^2}{K_2(a^2-r(p)^2)^2}-\frac{8(1+Ca)}{K_2(a^2-r(p)^2)}\right)\\
\leq & \dfrac{K_1}{K_2}-\frac{16}{K_2a^2}- \frac{8(1+Ca)}{K_2a^2}.
 \end{align*}
 So we obtain the estimate for $\rho$. Then for any $q\in D_a$, we get
 \begin{align*}
  \|df\|_0^2(q)\leq & \left(1-\dfrac{r^2(q)}{a^2}\right)^{-2} \left(\dfrac{K_1}{K_2}-\frac{16}{K_2a^2}- \frac{8(1+Ca)}{K_2a^2}\right)
 \end{align*}
Making $a\rightarrow +\infty$, we obtain $\|df\|_0^2(q)\leq \dfrac{K_1}{K_2}$, for any $q\in M$. The proof is complete.
\end{proof}
Now Corollary \ref{cor1.2} directly follows from the result and proof above. For Corollary \ref{cor1.3}, let $f:M\to M$ be any biholomorphism. Apply Theorem \ref{T1.1} to $f$ and $f^{-1} $ to get that $g\leq f^*g\leq g$. Therefore, $f$ is an isometry. As the isometry group of a compact Riemannian manifold is compact, the group of biholomorphisms is also compact, as a closed subset.

\end{document}